\documentclass[a4paper,11pt]{article}
\usepackage{amssymb,amsmath,amsthm,latexsym}
\usepackage{amsfonts}
	\usepackage{amsfonts}
\usepackage{graphicx}
\usepackage[pdftex,bookmarks,colorlinks=false]{hyperref}
\usepackage{verbatim}
\usepackage{caption}
\usepackage{subcaption}
\usepackage[hmargin=1.2in,vmargin=1.2in]{geometry}

\usepackage{fancyhdr}
\newtheorem{theorem}{Theorem}[section]

\newtheorem{definition}[theorem]{Definition}

\newtheorem{observation}[theorem]{Observation}

\newtheorem{proposition}[theorem]{Proposition}
\newtheorem{remark}[theorem]{Remark}

\title{\bf \sc Associated graphs of Certain Arithmetic IASI Graphs}
\author{{\bf N K Sudev \footnote{Department of Mathematics, Vidya Academy of Science \& Technology, Thalakkottukara, Thrissur - 680501, email: {\em sudevnk@gmail.com}}} and {\bf K A Germina\footnote{Department of Mathematics, School of Mathematical \& Physical Sciences, Central University of Kerala, Kasaragod, email:{\em srgerminaka@gmail.com}}}}
\date{}
\begin{document}
\maketitle

\begin{abstract}
Let $\mathbb{N}_0$ be the set of all non-negative integers and $\mathcal{P}(\mathbb{N}_0)$ be its power set. An integer additive set-indexer (IASI) is defined as an injective function $f:V(G)\to \mathcal{P}(\mathbb{N}_0)$ such that the induced function $f^+:E(G) \to \mathcal{P}(\mathbb{N}_0)$ defined by $f^+ (uv) = f(u)+ f(v)$ is also injective, where $f(u)+ f(v)$ is the sum set of the sets $f(u)$ and $f(v)$. A graph $G$ which admits an IASI is called an IASI graph. An arithmetic integer additive set-indexer is an integer additive set-indexer $f$, under which the set-labels of all elements of a given graph $G$ are the sets whose elements are in arithmetic progressions. In this paper, we discuss about admissibility of arithmetic integer additive set-indexers by certain associated graphs of the given graph $G$, like line graph, total graph, etc. 
\end{abstract}

\section{Introduction}

For all  terms and definitions, not defined specifically in this paper, we refer to \cite{FH} and for more about graph labeling, we refer to \cite{JAG}. Unless mentioned otherwise, all graphs considered here are simple, finite and have no isolated vertices. All sets mentioned in this paper are finite sets of non-negative integers. We denote the cardinality of a set $A$ by $|A|$. 

Let $\mathbb{N}_0$ denote the set of all non-negative integers and $\mathcal{P}(\mathbb{N}_0)$ be its power set. For all $A, B \subseteq \mathbb{N}_0$, the {\em sum set} of $A$ and $B$ is denoted by  $A+B$ and is defined as $A + B = \{a+b: a \in A, b \in B\}$.

\begin{definition}\label{D2}{\rm
\cite{GA} An {\em integer additive set-indexer} (IASI, in short) is defined as an injective function $f:V(G)\rightarrow \mathcal{P}(\mathbb{N}_0)$ such that the induced function $f^+:E(G) \rightarrow \mathcal{P}(\mathbb{N}_0)$ defined by $f^+ (uv) = f(u)+ f(v)$ is also injective.  A graph $G$ which admits an IASI is called an IASI graph.}
\end{definition}

\begin{definition}\label{D3}{\rm
The cardinality of the labeling set of an element (vertex or edge) of a graph $G$ is called the {\em set-indexing number} of that element.}
\end{definition}

In \cite{GS2}, the vertex set $V$ of a graph $G$ is defined to be {\em $l$-uniformly set-indexed}, if all the vertices of $G$ have the set-indexing number $l$.

By the term, an arithmetically progressive set, (AP-set, in short), we mean a set whose elements are in arithmetic progression. The common difference of the set-label of an element of $G$ is called the {\em deterministic index} of that element. 

\begin{definition}{\rm
\cite{GS7} An {\em arithmetic integer additive set-indexer} is an integer additive set-indexer $f$, under which the set-labels of all elements of a given graph $G$ are the sets whose elements are in arithmetic progressions. A graph that admits an arithmetic IASI is called an {\em arithmetic IASI graph}.

If all vertices of $G$ are labeled by the sets consisting of arithmetic progressions, but the set-labels of edges are not arithmetic progressions, then the corresponding IASI may be called {\em semi-arithmetic IASI}.}
\end{definition}

\begin{theorem}\label{T-AIASI-g}
\cite{GS7} A graph $G$ admits an arithmetic IASI if and only if for any two adjacent vertices in $G$, the deterministic index of one vertex is a positive integral multiple of the deterministic index of the other vertex and this positive integer is less than or equal to the cardinality of the set-label of the latter vertex.
\end{theorem}

\begin{proposition}\label{P-AIASI-1}
If the set-labels of both the end vertices of an edge have the same deterministic indices, say $d$, then the deterministic index of that edge is also $d$.
\end{proposition}

\begin{definition}{\rm
\cite{GS8} If the set-labels of all elements of a graph $G$ consist of arithmetic progressions with the same common difference $d$, then the corresponding IASI is called {\em isoarithmetic IASI}. That is, an arithmetic IASI of a graph $G$ is an isoarithmetic IASI if all elements of $G$ have the same deterministic index.}
\end{definition}

\begin{definition}{\rm
\cite{GS8} An arithmetic IASI $f$ of a graph $G$, under which the deterministic indices $d_i$ and $d_j$ of two adjacent vertices $v_i$ and $v_j$ respectively of $G$, holds the conditions $d_j=kd_i$ where $k$ is a non-negative integer such that $1< k \le |f(v_i)|$, is called {\em biarithmetic IASI}. If the value of $k$ is unique for all pairs of adjacent vertices of a biarithmetic IASI graph $G$, then that biarithmetic IASI is called {\em identical biarithmetic IASI} and $G$ is called an {\em identical biarithmetic IASI graph}.} 
\end{definition}

As we study the graphs, the set-labels of whose elements are AP-sets, all sets we consider in this discussion consists of at least three elements which are in ascending order.

In this paper, we investigate the admissibility of arithmetic integer additive set-indexers by certain graphs that are associated to a given graph $G$ and establish some results on arithmetic IASIs.

\section{Isoarithmetic IASIs of Associated Graphs}

In the following discussions, we study admissibility of isoarithmetic IASIs and biarithmetic IASIs by certain graphs associated to a given arithmetic IASI graph.

Throughout this section, we denote the set-label of a vertex $v_i$ of a given graph $G$ by $A_i$, which is a set of non-negative integers.  

\begin{proposition}
Let $G$ be an isoarithmetic IASI graph. Then, any non-trivial subgraph of $G$ is also an isoarithmetic IASI Graph.
\end{proposition}
\begin{proof}
let $f$ be an arithmetic IASI on $G$ and let $H\subset G$. The proof follows from the fact that the restriction $f|_H$ of $f$ to the subgraph $H$ is an induced isoarithmetic IASI on $H$.
\end{proof}

\begin{definition}{\rm
By {\em edge contraction operation} in $G$, we mean an edge, say $e$, is removed and its two incident vertices, $u$ and $v$, are merged into a new vertex $w$, where the edges incident to $w$ each correspond to an edge incident on either $u$ or $v$.}
\end{definition}

We establish the following theorem for the graphs obtained by contracting the edges of a given graph $G$. The following theorem verifies the admissibility of the graphs obtained by contracting the edges of a given isoarithmetic IASI graph $G$.

\begin{theorem}
Let $G$ be an isoarithmetic  IASI graph and let $e$ be an edge of $G$. Then, $G\circ e$ admits an isoarithmetic  IASI.
\end{theorem}
\begin{proof}
Let $G$ admits an  isoarithmetic IASI. Let $e$ be an edge in $E(G)$, the deterministic index of whose end vertices is $d$, where $d$ is a positive integer. Since $G$ is isoarithmetic IASI graph, the set set-label of each edge of $G$ is also an AP-set with the same common difference $d$. $G\circ e$ is the graph obtained from $G$ by deleting the edge $e$ of $G$ and identifying the end vertices of $e$. Label the new vertex thus obtained, say $w$, by the set-label of the deleted edge $e$. Then, each edge incident upon $w$ has a set-label which is also an AP-set with the same common difference $d$. Hence, $G\circ e$ is an isoarithmetic IASI graph. 
\end{proof}

\begin{definition}{\rm
\cite{KDJ} Let $G$ be a connected graph and let $v$ be a vertex of $G$ with $d(v)=2$. Then, $v$ is adjacent to two vertices $u$ and $w$ in $G$. If $u$ and $w$ are non-adjacent vertices in $G$, then delete $v$ from $G$ and add the edge $uw$ to $G-\{v\}$. This operation is known as an {\em elementary topological reduction} on $G$.}
\end{definition}

\begin{theorem}
Let $G$ be a graph which admits an isoarithmetic IASI. Then, any graph $G'$, obtained by applying a finite number of elementary topological reductions on $G$, also admits an isoarithmetic IASI. 
\end{theorem}
\begin{proof}
Let $G$ be a graph which admits an isoarithmetic IASI, say $f$. Then, all the elements of $G$ are labeled by AP-sets having the same common difference $d$, where $d$ is a positive integer. Let $v$ be a vertex of $G$ with $d(v)=2$. Then, $v$ is adjacent to two non-adjacent vertices $u$ and $w$ in $G$. Now remove the vertex $v$ from $G$ and introduce the edge $uw$ to $G-{v}$. Let $G'=(G-{v})\cup \{uw\}$. Now $V(G')\subset V(G)$. Let $f':V(G')\to \mathcal{P}(\mathbb{N}_0)$ such that $f'(v)=f(v)~ \forall ~v\in V(G')$ (or $V(G)$) and the associated function $f'^+:E(G')\to \mathcal{P}(\mathbb{N}_0)$ defined by 
\[ f'^+(e)= \left\{
\begin{array}{l l}	
f^+(e)& \quad \text{if $e\ne uw$}\\
f(u)+f(w)& \quad \text{if $e=uw$}
\end{array} \right.\]
Clearly, $f'$ is an isoarithmetic IASI of $G'$. 
\end{proof}

Another associated graph of a given graph $G$ is its graph subdivision. The notion of a graph subdivision is given below and its admissibility of arithmetic IASI are established in the following theorem.

\begin{definition}{\rm
\cite{RJT} A {\em subdivision} of a graph $G$ is the graph obtained by adding vertices of degree two into some or all of its edges.}
\end{definition}

\begin{theorem}
The graph subdivision $G^{\ast}$ of an isoarithmetic IASI graph $G$ also admits an isoarithmetic IASI.
\end{theorem}
\begin{proof}
Let $u$ and $v$ be two adjacent vertices in $G$. Since $G$ admits an isoarithmetic IASI, the set-labels of two vertices $u$, $v$ and the edge $uv$ of $G$ are AP-sets with the common difference $d$, where $d$ is a positive integer. Introduce a new vertex $w$ to the edge $uv$. Now, we have two new edges $uw$ and $vw$ in place of $uv$. Extend the set-labeling of $G$ by labeling the vertex $w$ by the same set-label of the edge $uv$. Then, both the edges $uw$ and $vw$ have the set-labels which are AP-sets with the same common difference $d$. Hence, $G^{\ast}$ admits an isoarithmetic IASI. 
\end{proof}

Recall the following definition of line graph of a graph.

\begin{definition}{\rm
\cite{DBW} For a given graph $G$, its line graph $L(G)$ is a graph such that  each vertex of $L(G)$ represents an edge of $G$ and two vertices of $L(G)$ are adjacent if and only if their corresponding edges in $G$ are incident on a common vertex in $G$.}
\end{definition}

An interesting question we need to address here is whether the line graph of an isoarithmetic IASI graph admits an isoarithmetic IASI. The following theorem answers this question.

\begin{theorem}
If $G$ is an isoarithmetic IASI graph, then its line graph $L(G)$ is also an isoarithmetic IASI graph.
\end{theorem}
\begin{proof}
Since $G$ is an isoarithmetic IASI graph, the elements of $G$ have the set-labels whose elements are in arithmetic progression with the same common difference, say $d$, where $d$ is a positive integer. Label each vertex of $L(G)$ by the same set-label of its corresponding edge in $G$. Hence, the set-labels of all vertices in $L(G)$ are AP-sets with the same common difference $d$. Therefore, the set-labels of all edges of $L(G)$ are  also AP-sets with the same common difference $d$. That is, $L(G)$ is also an isoarithmetic graph.
\end{proof}

\begin{definition}{\rm
\cite{MB} The {\em total graph} of a graph $G$, denoted by $T(G)$, is the graph having the property that a one-to one correspondence can be defined between its points and the elements (vertices and edges) of $G$ such that two points of $T(G)$ are adjacent if and only if the corresponding elements of $G$ are adjacent (if both elements are edges or if both elements are vertices) or they are incident (if one element is an edge and the other is a vertex). }
\end{definition} 

\begin{theorem}
If $G$ is an isoarithmetic IASI graph, then its total graph $T(G)$ is also an isoarithmetic IASI graph.
\end{theorem}
\begin{proof}
Let the graph $G$ admits an isoarithmetic IASI, say $f$. Then, for any element (a vertex or an edge) $x$ of $G$, the set-label $f(x)$ is an AP-set of non-negative integers with the common difference, say $d$, $d$ being a positive integer. Define a map $f':V(T(G))\to \mathcal{P}(\mathbb{N}_0)$ which assigns the same set-labels of the corresponding elements in $G$ under $f$ to the vertices of $T(G)$. Clearly, $f'$ is injective and for each vertex $u_i$ in $T(G)$, $f'(u_i)$ is an AP-set with the same common difference $d$. Now, define the associated function $f^+:E(T(G))\to \mathcal{P}(\mathbb{N}_0)$ defined by $f'^+(u_iu_j)= f'(u_i)+f'(u_j),~ u_i,u_j\in V(T(G))$. Then, $f'^+$ is injective and each $f'^+(u_iu_j)$ is also an AP-set with the same common difference $d$. Therefore, $f'$ is an isoarithmetic IASI of $T(G)$. This completes the proof.
\end{proof}

\section{Biarithmetic IASI of Associated Graphs}


In this section, we discuss the admissibility of biarithmetic IASIs by the associated graphs of a given biarithmetic IASI graph.

\begin{theorem}\label{T-CUIASI2}
\cite{GS8} A biarithmetic IASI of a graph $G$ is an $l$-uniform IASI if and only if $G$ has $p$ bipartite components, where $p$ is the number of distinct pair $(m_i,n_j)$ of positive integers such that $m_i$ and $n_j$ are the set-indexing numbers of adjacent vertices in $G$ and $l=m_i+n_j-1$.
\end{theorem}

What are the characteristics of the line graph of a biarithmetic graph? The following results provide a solution to the problem.

\begin{theorem}
Let $G$ be a biarithmetic IASI graph. Then, its line graph $L(G)$ admits an isoarithmetic IASI if and only if $G$ is bipartite.
\end{theorem}
\begin{proof}
Let $G$ be a bipartite graph which admits a biarithmetic IASI, with the bipartition $(X,Y)$. Since $G$ admits a biarithmetic IASI, there exists an integer $k>1$ such that the vertices of $X$ are labeled by distinct AP-sets of non-negative integers with common difference $d$ and the vertices of $Y$ are labeled by distinct AP-sets of non-negative integers with common difference $kd$.  Then, the set-label of every edge of $G$ is also an AP-set with the common difference $d$. Therefore, the set-labels of all vertices in $L(G)$ are AP-sets with the same common difference $d$. Hence, every edge of $L(G)$ also has a set-label which is an AP-sets with the same common difference $d$. That is, $L(G)$ admits an isoarithmetic IASI.

Conversely, let $L(G)$ is an isoarithmetic IASI graph. Hence, every element of $L(G)$ must be labeled by an AP-set with common difference $d$. Therefore, the all the edges in $G$ must have set labels which are AP-sets with the same common difference $d$. Since, $G$ admits a biarithmetic IASI, the set-label of one end vertex of every edge must be an AP-set with common difference $d$ and the set-label of the other end vertex is an AP-set with the common difference $kd$. Let $X$ be the set of all vertices of $G$ which are labeled by the AP-sets with common difference $d$ and $Y$ be the set of all vertices of $G$ labeled by the AP-sets with common difference $kd$. Since $k>1$, no two vertices in $X$ can be adjacent to each other and no two vertices in $Y$ can be adjacent to each other. Therefore, $(X,Y)$ is a bipartition of $G$. Hence, $G$ is bipartite. This completes the proof.
\end{proof}

\begin{theorem}\label{T-kAIASI1}
If the line graph $L(G)$ of a biarithmetic IASI graph $G$ admits a biarithmetic IASI, then $G$ is acyclic.
\end{theorem}
\begin{proof}
Assume that $L(G)$ is a biarithmetic IASI graph. If possible, let $G$ contains a cycle $C_n=v_1v_2v_3\ldots v_nv_1$. Let $e_i=v_iv_{i+1}, 1\le i \le n$ and let $u_i$ be the vertex in $L(G)$ corresponding to the edge $e_i$ in $G$. Label each vertex $v_i$ of $G$ by the set whose elements are arithmetic progression with common difference $d_i$ where $d_{i+1}=k.d_i;~ k\ge |f(v_i)|_{min}$.  Without loss of generality, let $f(v_1)$ has the minimum cardinality. Since $L(G)$ admits a biarithmetic IASI, adjacent vertices $u_i$ and $u_{i+1}$ in $L(G)$ are labeled by the sets whose elements are in arithmetic progressions whose common differences are $d_i$ and $d_{i+1}=k.d_i$ respectively. Therefore, the corresponding edges $e_i$ and $e_{i+1}$ of $G$ must also have the same set-labeling. Hence, alternate vertices of $G$ can not have the set-labels with the same common difference. Then, $d_i=k^i.d_1, 1<k\le |f(v_1)|$. Here, we notice that the set-label of one end vertex $v_n$ of the edge $v_nv_1$ in the cycle $C_n$ has the common difference $k^n.d_1$ and the set-label of  other end vertex $v_1$ has the common difference $d_1$, which is a contradiction to the fact that $G$ is biarithmetic IASI graph. Therefore, G is acyclic.
\end{proof}

\begin{remark}{\rm
The converse of the theorem need not be true. For example, the graph $K_{1,3}$ admits a biarithmetic IASI and is acyclic, but its line graph does not admit a biarithmetic IASI.}
\end{remark}

What is the condition required for an acyclic graph to admit a biarithmetic IASI? The following theorem establishes the necessary and sufficient condition for a biarithmetic IASI graph to have its line graph, a biarithmetic IASI graph.

\begin{theorem}\label{T-kAIASI2}
The line graph of a biarithmetic IASI graph admits a biarithmetic IASI if and only if $G$ is a path.
\end{theorem} 
\begin{proof}
The necessary part of the theorem follows from Theorem \ref{T-kAIASI1}. Conversely, assume that $G$ is a path. Let $G=v_1v_2v_3\ldots v_n$. Label the vertex $v_i$ by an AP-set with the common difference $d_i$, where $k\le |f(v_i)|_{min}$. Without loss of generality, let $f(v_1)$ has the minimum cardinality. Then, $d_i=k^i.d_1, 1<k\le |f(v_1)|$. Then, the set-label of each edge $e_i$ of $G$ is an AP-set with difference $d_i=k.d_{i-1}$. Hence, the each vertex $u_i$ in $L(G)$ corresponding to the edge $e_i$ has the set-label which is an AP-set with the common difference $d_i=k.d_{i-1}=k^{i-1}.d_1$. Hence, $L(G)$ admits a biarithmetic IASI. This completes the proof. 
\end{proof}

In the above theorems, the value of $k$ should be within $1$ and $|f(v_i)|min$, the minimum among the set-indexing numbers of the vertices of $G$. A question that arises in this context is about the validity of these results if $k>|f(v_i)|min$. The following result answers this question.

\begin{theorem}
Let $G$ admits a biarithmetic IASI $f$. Let $k=\frac{f(v_i)}{f(v_j)}$ for any two adjacent vertices of $G$. If $k> \min (|f(v_i)|)$, then the line graph of $G$ does not admit an arithmetic IASI. 
\end{theorem}
\begin{proof}
Let $G$ admits a biarithmetic IASI $f$ and let $V(G)=\{v_1,v_2,v_3,\ldots,\\ v_n\}$ be the vertex set of $G$. If possible, let $k>\min (|f(v_i)|)$. Then, the set-label of the edge $v_iv_{i+1}$ will not be an AP-set. That is, $f$ is a semi-arithmetic IASI. Therefore, the set-label of the vertex $u_l$ of its line graph $L(G)$ corresponding to the edge $v_iv_{i+1}$ in $G$ is not an AP-set. Hence, for $k>|f(v_i)|_{min};~ v_i\in V(G)$, the line graph $L(G)$ of a biarithmetic IASI graph does not admit an arithmetic IASI.
\end{proof}

\begin{theorem}
The total graph of an identical biarithmetic IASI graph is an arithmetic IASI graph.  
\end{theorem}
\begin{proof}
The vertices of $T(G)$ corresponding to the vertices of $G$ have the same set-labels and the edges in $T(G)$ connecting these vertices also preserve the same set-labels of the corresponding edges of $G$. The vertices of $T(G)$ corresponding to the edges of $G$ are given the same set-labels of the corresponding set-labels of the edges of $G$. Hence, all these vertices in $T(G)$ have the same deterministic index, say $d$, and hence the edges in $T(G)$ connecting these vertices also have the same deterministic index $d$. As the deterministic index of an edge and one of its end vertex are the same and the deterministic index of the other end vertex is a positive integral multiple of the deterministic index of the edge, where this integer is less than or equal to the cardinality of the set-label of the other end vertex,  the edges corresponding to the incidence relations in $G$ also have the deterministic index $d$. Hence, $T(G)$ admits an arithmetic IASI.
\end{proof}

\begin{theorem}
The total graph of a biarithmetic IASI graph is an arithmetic IASI graph.  
\end{theorem}
\begin{proof}
The vertices of $T(G)$ corresponding to the vertices of $G$ have the same set-labels and the vertices of $T(G)$ corresponding to the edges of $G$ are given the same set-labels of the corresponding set-labels of the edges of $G$. Also, the deterministic index of an edge and one of its end vertex are the same and the deterministic index of the other end vertex is a positive integral multiple of the deterministic index of the other end vertex, where this integer is less than or equal to the cardinality of the set-label of the other end vertex. Hence, for every two adjacent vertices in $T(G)$, the deterministic index of one is a positive integral multiple of of the deterministic index of the other, where this integer is less than or equal to the set-indexing number of the latter. Therefore, by Theorem \ref{T-AIASI-g}, $T(G)$ is an arithmetic IASI graph. 
\end{proof}

The following theorem checks whether the total graph corresponding to a biarithmetic IASI graph $G$ admits a biarithmetic IASI.

\begin{theorem}
The total graph of a biarithmetic IASI graph is not a biarithmetic IASI graph.  
\end{theorem}
\begin{proof}
We observe that every edge in $G$ corresponds to a triangle $K_3$ in its total graph. Since $K_3$ can not admit a biarithmetic IASI, $T(G)$ is not a biarithmetic IASI graph.
\end{proof}

If $G$ is an identical biarithmetic IASI graph, will $G\circ e, ~e\in E(G)$ be an identical biarithmetic IASI graph? The cycle $C_4$ is a identical biarithmetic graph, but for any edge $e$ of $C_4$, $C_4\circ e=C_3$, which does not admit an identical biarithmetic IASI. Hence, we observe 

\begin{observation} 
A graph obtained from an identical biarithmetic IASI graph by contracting an edge of it, need not a biarithmetic IASI graph.
\end{observation}

We also prove a similar for the graphs obtained from a biarithmetic IASI graph by a finite number of topological reductions. 

\begin{proposition}
Let $H$ be a graph obtained by finite number of topological reduction on a biarithmetic IASI graph $G$. Then, $H$ is not a biarithmetic IASI graph.
\end{proposition}
\begin{proof}
Let $v$ be a vertex of $G$ with degree $2$. Without loss of generality, let the set-label of $v$ be an AP-set with difference $d$. Let $u$ and $w$ be the adjacent vertices of $v$ which are not adjacent to each other. Since $G$ is a biarithmetic graph, both $u$ and $w$ must be labeled by distinct AP-sets with difference $k.d$.
Now delete the vertex $v$ and join $u$ and $w$. Let $H=(G-\{v\})\cup \{uw\}$. Then, both the end vertices of the edge $vw$ has the set labels which are AP-sets of the same difference $k.d$. Hence, $H$ does not admit a biarithmetic IASI.
\end{proof}

\begin{theorem}
The graph  subdivision $G^{\ast}$ of a given biarithmetic IASI graph $G$ does not admit a biarithmetic IASI.
\end{theorem}
\begin{proof}
Let $u$ and $v$ be two adjacent vertices in $G$ whose set-labels are AP-sets with common differences $d$ and $k.d$ respectively. Since $G$ admits a biarithmetic IASI, the set-label of the edge $uv$ is AP-set with different difference $d$. If we introduce a new vertex $w$ to the edge $uv$ and extend the set-labeling of $G$ by labeling the vertex $w$ by the same set-label of the edge $uv$, then, the set-labels of both $u$ and $w$ (or $v$ and $w$) are AP-sets with the same difference $d$. Hence, $G^{\ast}$ does not admit a biarithmetic IASI. 
\end{proof}

\section{Further Points of Discussions}

In this section we make some remarks on semi-arithmetic IASI graphs and their associated graphs. We observe that  if the set labels of all vertices of $G$ are AP-sets with distinct differences, then the set-labels of edges will not be AP-sets. Hence, We have the following observations.

\begin{proposition}
The line graph $L(G)$ of a semi-arithmetic IASI graph $G$ does not admit an arithmetic IASI (or a semi-arithmetic IASI).
\end{proposition}

\begin{proposition}
The Total graph $T(G)$ of a semi-arithmetic IASI graph $G$ does not admit an arithmetic IASI (or a semi-arithmetic IASI).
\end{proposition}

From the fact that a graph $G$, its subdivision graph, the graph obtained by contracting an edge and the graph obtained by elementary topological reductions have some common edges, we observe the following results.

\begin{proposition}
The graph $G\circ e$, obtained by contracting an edge $e$ of a semi-arithmetic IASI graph $G$,does not admit an arithmetic IASI (or a semi-arithmetic IASI). 
\end{proposition}

\begin{proposition}
The subdivision graph $G^{\ast}$ of a semi-arithmetic IASI graph does not admit an arithmetic IASI (or a semi-arithmetic IASI).
\end{proposition} 

\begin{proposition}
The graph $G'$, obtained by applying elementary topological reduction on a semi-arithmetic IASI graph $G$, does not admit an arithmetic IASI (or a semi-arithmetic IASI). 
\end{proposition} 

\section{Conclusion}

In this paper, we have discussed some characteristics of certain graphs associated a given graph which admits certain types of arithmetic IASI. We have formulated some conditions for those graph classes to admit arithmetic IASIs. Here, we have discussed about isoarithmetic IASI graphs and biarithmetic IASI graphs only. The existence of similar results for semi-arithmetic IASI graphs are yet to be studied. 

The IASIs under which the vertices of a given graph are labeled by different standard sequences of non negative integers, are also worth studying.   The problems of establishing the necessary and sufficient conditions for various graphs and graph classes to have certain IASIs still remain unsettled. All these facts highlight a wide scope for further studies in this area.

\end{document}